\documentclass[a4paper,12pt]{article}
\usepackage{amsmath}
\usepackage{amsthm}
\usepackage{amssymb}
\usepackage{enumerate}
\usepackage{url}
\usepackage[utf8]{inputenc}

\title{Integral representations for the Hartman--Watson density}

\author{Yuu Hariya\thanks{Supported in part by JSPS KAKENHI Grant Number 17K05288}}

\date{\empty}
\numberwithin{equation}{section}

\theoremstyle{plain}

\newtheorem{thm}{Theorem}[section]

\newtheorem{prop}{Proposition}[section]

\newtheorem{lem}{Lemma}[section]

\theoremstyle{definition}

\theoremstyle{remark}
\newtheorem{rem}{Remark}[section]

\begin{document}

\newcommand\ND{\newcommand}
\newcommand\RD{\renewcommand}

\ND\N{\mathbb{N}}
\ND\R{\mathbb{R}}
\ND\Q{\mathbb{Q}}
\ND\C{\mathbb{C}}

\ND\F{\mathcal{F}}

\ND\kp{\kappa}

\ND\ind{\boldsymbol{1}}

\ND\al{\alpha }
\ND\la{\lambda }
\ND\La{\Lambda }
\ND\ve{\varepsilon}
\ND\Om{\Omega}

\ND\ga{\gamma}

\ND\lref[1]{Lemma~\ref{#1}}
\ND\tref[1]{Theorem~\ref{#1}}
\ND\pref[1]{Proposition~\ref{#1}}
\ND\sref[1]{Section~\ref{#1}}
\ND\ssref[1]{Subsection~\ref{#1}}
\ND\aref[1]{Appendix~\ref{#1}}
\ND\rref[1]{Remark~\ref{#1}} 
\ND\cref[1]{Corollary~\ref{#1}}
\ND\eref[1]{Example~\ref{#1}}
\ND\fref[1]{Fig.\ {#1} }
\ND\lsref[1]{Lemmas~\ref{#1}}
\ND\tsref[1]{Theorems~\ref{#1}}
\ND\dref[1]{Definition~\ref{#1}}
\ND\psref[1]{Propositions~\ref{#1}}
\ND\rsref[1]{Remarks~\ref{#1}}
\ND\sssref[1]{Subsections~\ref{#1}}

\ND\pr{\mathbb{P}}
\ND\ex{\mathbb{E}}
\ND\br{W}

\ND\prb[2]{P^{(#1)}_{#2}}
\ND\exb[2]{E^{(#1)}_{#2}}

\ND\eb[1]{e^{B_{#1}}}
\ND\ebm[1]{e^{-B_{#1}}}
\ND\hbe{\Hat{\beta}}
\ND\hB{\Hat{B}}
\ND\argsh{\mathrm{Argsh}\,}
\ND\zmu{z_{\mu}}
\ND\GIG[3]{\mathrm{GIG}(#1;#2,#3)}
\ND\gig[3]{I^{(#1)}_{#2,#3}}
\ND\argch{\mathrm{Argch}\,}

\ND\vp{\varphi}
\ND\eqd{\stackrel{(d)}{=}}
\ND\db[1]{B^{(#1)}}
\ND\da[1]{A^{(#1)}}

\ND\anu{\alpha ^{(\nu )}}

\ND\Ga{\Gamma}
\ND\calE{\mathcal{E}}
\ND\calD{\mathcal{D}}

\ND\f{F}
\ND\g{G}

\def\thefootnote{{}}

\maketitle 
\begin{abstract}
This paper concerns the density of the Hartman--Watson law. 
Yor (1980) obtained an integral formula that gives a closed-form 
expression of the Hartman--Watson density. In this paper, based on 
Yor's formula, we provide alternative integral representations for the 
density. As an immediate application, we recover 
in part a Dufresne's result (2001) that exhibits remarkably 
simple representations for the laws of exponential additive 
functionals of Brownian motion.
\footnote{Mathematical Institute, Tohoku University, Aoba-ku, Sendai 980-8578, Japan}
\footnote{E-mail: hariya@tohoku.ac.jp}
\end{abstract}

\section{Introduction}\label{;intro}

Let $B=\{ B_{t}\} _{t\ge 0}$ be a one-dimensional standard Brownian motion. 
For every $\mu \in \R $, we denote by 
$\db{\mu }=\{ \db{\mu }_{t}:=B_{t}+\mu t\} _{t\ge 0}$ the Brownian motion 
with constant drift $\mu $ and set 
\begin{align*}
 \da{\mu }_{t}:=\int _{0}^{t}e^{2\db{\mu }_{s}}ds, \quad t\ge 0;  
\end{align*}
when $\mu =0$, we simply write $A_{t}$ for $\da{\mu }_{t}$. 
This additive functional, together with the geometric Brownian motion 
$e^{\db{\mu }_{t}},\,t\ge 0$, plays an important role in a number of areas 
such as option pricing in mathematical finance, 
diffusion processes in random environments, probabilistic study of 
Laplacians on hyperbolic spaces, and so on; see the detailed surveys 
\cite{mySI,mySII} by Matsumoto--Yor and references therein. 

In \cite{yor92}, Yor proved that for every $t>0$, the joint law 
of $B_{t}$ and $A_{t}$ is given by 
\begin{align}\label{;jl1}
 \pr \!\left( 
 B_{t}\in dx,\,A_{t}\in dv
 \right) 
 =\frac{1}{v}\exp \left\{ 
 -\frac{1}{2v}\left( 1+e^{2x}\right) 
 \right\} \Theta (e^{x}/v,t)\,dxdv,\quad x\in \R ,\,v>0, 
\end{align}
or equivalently, 
\begin{align}\label{;jl2}
 \pr \left( \eb{t}\in du,\,A_{t}\in dv\right) 
 =\frac{1}{uv}\exp \left( -\frac{1+u^{2}}{2v}\right) 
 \Theta (u/v,t)\,dudv, \quad u,v>0, 
\end{align}
where for every $r>0$, the function $\Theta (r,t),\,t>0$, is the 
(unnormalized) density of the so-called Hartman--Watson distribution 
(\cite{hw}) which is characterized by the Laplace transform 
\begin{align}\label{;LTt}
 \int _{0}^{\infty }dt\,
 \exp \left( -\frac{\la ^{2}}{2}t\right) 
 \Theta (r,t)
 =I_{|\la |}(r),\quad \la \in \R . 
\end{align}
Here for every index $\nu \in \R $, 
the function $I_{\nu }$ is the modified Bessel function 
of the first kind of order $\nu $; see \cite[Section~5.7]{leb} 
for the definition. By the Cameron--Martin relation, we see from 
\eqref{;jl2} that for every $\mu \in \R $ and $t>0$, the law of 
$\da{\mu }_{t}$ is expressed as 
\begin{align}
 \frac{\pr (\da{\mu }_{t}\in dv)}{dv}
 &=\frac{1}{v}\exp \left( 
 -\frac{1}{2v}-\frac{\mu ^{2}}{2}t
 \right) \int _{0}^{\infty }\frac{du}{u}\,u^{\mu }\exp \left( 
 -\frac{u^{2}}{2v}\right) \Theta (u/v,t)\notag \\
 &=v^{\mu -1}\exp \left( 
 -\frac{1}{2v}-\frac{\mu ^{2}}{2}t
 \right) \int _{0}^{\infty }dr\,r^{\mu -1}\exp \left( 
 -\frac{v}{2}r^{2}
 \right) \Theta (r,t),\quad v>0, \label{;Amlaw0}
\end{align}
where for the second line, we changed the variables with $u=vr,\,r>0$. 

It is also proven by Yor \cite{yor80} that the function $\Theta $ 
admits the following integral representation: for every $r>0$ and $t>0$, 
\begin{align}\label{;irepr0}
 &\Theta (r,t)=\frac{r}{\sqrt{2\pi ^3t}}
 \int _{0}^\infty dy\,\exp \left( 
 \frac{\pi ^2-y^2}{2t}
 \right) \exp \left( 
 -r\cosh y
 \right) \sinh y\sin \left( \frac{\pi y}{t}\right) , 
\end{align}
which may be rephrased conveniently as 
\begin{align}\label{;irepr0d}
  \Theta (r,t)=
  \frac{r}{2\pi }\exp \left( 
  \frac{\pi ^{2}}{2t}
  \right) 
  \ex \!\left[ 
  \exp \left( -r\cosh B_{t}\right) \sinh B_{t}
  \sin \left( \frac{\pi }{t}B_{t}\right) 
  \right] ,
\end{align}
by the fact that the function $\R \ni y\mapsto \sinh y\sin (\pi y/t)$ is 
symmetric. We also refer to \cite[p.~82]{bs} for \eqref{;irepr0}, 
and \cite[p.~175, formula~1.10.8]{bs} for the joint law 
\eqref{;jl1}. 

The function $\Theta $ as well as formula \eqref{;irepr0} 
have continuously attracted researchers' attention particularly from 
the view point of evaluation of Asian options in the 
Black--Scholes framework; see, e.g., \cite{sch, cs, bm, lya} and 
references therein.  Yor obtained \eqref{;irepr0} by inverting the 
Laplace transform \eqref{;LTt}, reasoning of which is also 
reproduced in \cite[Appendix~A]{mySI}. 
In \cite[Subsection~A.3]{har}, we explain 
\eqref{;irepr0} via 
\begin{align}\label{;LTr}
 \int _{0}^{\infty }\frac{dr}{r}\,e^{-r\cosh x}\Theta (r,t)
 =\frac{1}{\sqrt{2\pi t}}\exp \left( 
  -\frac{x^{2}}{2t}
  \right) ,\quad t>0,\, x\in \R , 
\end{align}
which relation is found in, e.g., \cite[Proposition~4.5(i)]{myPI} and, 
as was observed in \cite[Proposition~4.2]{mySI}, may be obtained 
by integrating both sides of \eqref{;jl1} with respect to $v$. 
In this paper, we continue our discussion of \cite[Subsection~A.3]{har} 
and, based on Yor's formula \eqref{;irepr0} (or \eqref{;irepr0d}), 
aim at providing the following alternative representations of $\Theta $: 

\begin{thm}\label{;mt}
For every $r>0$ and $t>0$, it holds that 
\begin{align}
  \Theta (r,t)&=\frac{r}{\pi }\exp \left( \frac{\pi ^{2}}{8t}\right) 
  \ex \!\left[ 
  \cosh B_{t}\cos (r\sinh B_{t})\cos \left( 
  \frac{\pi }{2t}B_{t}
  \right) 
  \right] \label{;irepr1}\\
  &=\frac{r}{\pi }\exp \left( \frac{\pi ^{2}}{8t}\right) 
  \ex \!\left[ 
  \cosh B_{t}\sin (r\sinh B_{t})\sin \left( 
  \frac{\pi }{2t}B_{t}
  \right) 
  \right] \label{;irepr2}\\
  &=\frac{r}{2\pi }\exp \left( \frac{\pi ^{2}}{8t}\right) 
  \ex \!\left[ 
  \cosh B_{t}\cos \left( r\sinh B_{t}-
  \frac{\pi }{2t}B_{t}
  \right) 
  \right] . \label{;irepr3}
 \end{align}
 More generally, we have for every $r>0$ and $t>0$, 
 \begin{align}\label{;irepr4}
  \Theta (r,t)&=\frac{r}{\pi }\exp \left( \frac{\pi ^{2}}{8t}\right) 
  \ex \!\left[ 
  \cosh B_{t}\cos (r\sinh B_{t}-\nu )\cos \left( 
  \frac{\pi }{2t}B_{t}-\nu 
  \right) 
  \right] , 
 \end{align}
 where $\nu \in \R $ is arbitrary. Representations 
 \eqref{;irepr1}, \eqref{;irepr2} and \eqref{;irepr3} may be seen as 
 the case $\nu =0,\pi /2,\pi /4$, respectively. 
\end{thm}

The third representation \eqref{;irepr3} follows readily from 
\eqref{;irepr1} and \eqref{;irepr2} by summing them. 
If we multiply \eqref{;irepr1} and \eqref{;irepr2} by 
$\cos ^{2}\nu $ and $\sin ^{2}\nu $, respectively, then taking 
their sum leads to the fourth representation \eqref{;irepr4}; 
for details, see the proof of \tref{;mt} given in \sref{;prfmt}
whose reasoning will also reveal that $\nu $ may be replaced by any 
complex number.  

\begin{rem}\label{;rdiff}
\thetag{1} Taking the difference of \eqref{;irepr1} and \eqref{;irepr2} 
leads to the following fact of interest: 
\begin{align}\label{;diff0}
 \ex \!\left[ 
  \cosh B_{t}\cos \left( r\sinh B_{t}+
  \frac{\pi }{2t}B_{t}
  \right) 
  \right] =0 \quad \text{for any $r>0$ and $t>0$,}
\end{align}
which will be returned to in \rref{;rcontr} and recalled in 
the proof of \pref{;pother} below. Representation \eqref{;irepr3} 
should be considered jointly with \eqref{;diff0}. 

\noindent 
\thetag{2} Differentiating representation \eqref{;irepr4} with 
respect to $\nu $ yields 
\begin{align*}
  \ex \!\left[ 
  \cosh B_{t}\sin \left( 
  r\sinh B_{t}+\frac{\pi }{2t}B_{t}-2\nu 
  \right) 
  \right] =0,\quad r>0,\,t>0, 
\end{align*}
for any $\nu \in \R $, which agrees with \eqref{;diff0} since 
for any $r>0$ and $t>0$, 
\begin{align*}
 \ex \!\left[ 
  \cosh B_{t}\sin \left( r\sinh B_{t}+
  \frac{\pi }{2t}B_{t}
  \right) 
  \right] =0
\end{align*}
due to the fact that the function 
$
 \sin \left( r\sinh x+
 \pi x/(2t)
 \right) ,\,x\in \R 
$, is an odd function. 

\noindent 
\thetag{3} In view of \eqref{;irepr0}, we see from \eqref{;irepr1} that 
\begin{align*}
 \lim _{r\to \infty }
 \ex \!\left[ 
  \cosh B_{t}\cos (r\sinh B_{t})\cos \left( 
  \frac{\pi }{2t}B_{t}
  \right) 
  \right] =0, 
\end{align*}
which may also be explained by the Riemann--Lebesgue lemma. 
The same remark is true for \eqref{;irepr2}. 

\noindent 
\thetag{4} In a recent paper \cite{jw19} by Jakubowski and 
Wi\'sniewolski, they have obtained in their Theorem~3.7 
and Corollary~3.9 the third representation \eqref{;irepr3}, 
appealing to the fact due to 
Alili and Gruet (\cite[equation~(1.5)]{ag}) that when $t>0$, the density 
of the law of $A_{t}$ is given by 
\begin{align*}
 \frac{1}{\sqrt{2\pi t}}\int _{\R }dy\,
 \frac{\cosh y}{\sqrt{2\pi v^{3}}}
 \exp \left( 
 -\frac{\cosh ^{2}y}{2v}
 \right) \exp \biggl\{ 
 -\frac{\left( y+\sqrt{-1}\pi /2\right) ^{2}}{2t}
 \biggr\} 
\end{align*}
for $v>0$. As noted above, representation \eqref{;irepr3} is 
immediate from \eqref{;irepr1} and \eqref{;irepr2}; our method used 
in deriving those two representations quite differs from that of 
\cite{jw19} and hinges upon a simple observation exhibited in 
\lref{;lequiv} below, which explains the Riemann--Lebesgue lemma 
in a clear fashion as to some expectations relative to $B_{t}$. 
In view of the injectivity of Mellin transform, 
representation \eqref{;irepr3} may also be deduced from the 
two representations of $\al ^{(\nu, \ve )}_{t}(w)$ given just 
after Theorem~2.1 of \cite{sch} by Schr\"oder.

\noindent 
\thetag{5} In \pref{;pother}, we also present other relevant 
integral representations of $\Theta $. 
\end{rem}

\tref{;mt} has several applications. One of its immediate consequences 
is that for every fixed $t>0$, the derivative of $\Theta (r,t)$ of any order 
at $r=0+$ vanishes due to repeated occurrence of the multiple factor 
$\sin (r\sinh B_{t})$ in integrands when taking derivatives 
of \eqref{;irepr1} and \eqref{;irepr2} with respect to $r$. 

\begin{prop}\label{;pderi} Fix $t>0$. It holds that 
\begin{align}\label{;qpderi1}
 \lim _{r\to 0+}\frac{\partial ^{n}}{\partial r^{n}}\Theta (r,t)=0, \quad 
 n=0,1,2,\ldots . 
\end{align}
In particular, as $r\to 0+$, 
\begin{align}\label{;qpderi2}
 \Theta (r,t)=o(r^{\kappa }) \quad \text{for any $\kappa >0$.}
\end{align}
\end{prop}
For the proof, see \ssref{;prfpderi}; the above fact \eqref{;qpderi1} 
can also be deduced readily from \eqref{;irepr3} and \eqref{;diff0}, 
for which we refer the reader to \rref{;rcontr}. 
As a consequence of \eqref{;qpderi2}, 
we see that the integral in \eqref{;Amlaw0} does converge even when 
$\mu \le 0$. We recall that \eqref{;qpderi1} is also observed in 
\cite[Subsection~2.1]{my2003} from Yor's formula 
\eqref{;irepr0}, combined with a remark by Stieltjes in 1894 stating 
that for any integer $n$, 
\begin{align}\label{;stieltjes}
 \int _{\R }dx\,\exp \left( 
 -\frac{x^{2}}{2t}
 \right) e^{nx}\sin \left( 
 \frac{\pi x}{t}
 \right) =0
\end{align}
(see \cite[equation~(6.3)]{my2003}). Our \tref{;mt} allows us to obtain 
\eqref{;qpderi1} more directly, without relying on Stieltjes' remark. 

\begin{rem}\label{;rstieltjes}
 Since the left-hand side of \eqref{;stieltjes} is written as 
 \begin{align*}
  \sqrt{2\pi t}\,\ex \!\left[ 
  e^{nB_{t}}\sin \left( 
  \frac{\pi }{t}B_{t}
  \right) 
  \right] , 
 \end{align*}
 \eqref{;stieltjes} is verified by the Cameron--Martin relation: 
 \begin{align*}
  e^{-n^{2}t/2}\,\ex \!\left[ 
  e^{nB_{t}}\sin \left( 
  \frac{\pi }{t}B_{t}
  \right) 
  \right] &=
  \ex \!\left[ 
  \sin \left\{ 
  \frac{\pi }{t}(B_{t}+nt)
  \right\} 
  \right] \\
  &=(-1)^{n}\ex \!\left[ 
  \sin \left( 
  \frac{\pi }{t}B_{t}
  \right) 
  \right] , 
 \end{align*}
 which is zero by the symmetry of Brownian motion. 
\end{rem}

When inserting representation \eqref{;irepr0} into 
\eqref{;Amlaw0}, a double integral emerges in the description of 
the law of $\da{\mu }_{t}$. The second application of \tref{;mt} is 
that, when $\mu $ is a nonnegative integer, we easily reduce 
that apparently complicated double integral to a single integral 
by Fubini's theorem, thanks to the well-known formulae 
(see \cite[equations~(4.11.2) and (4.11.3)]{leb}) for the Hermite 
polynomials 
\begin{align*}
 H_{2n}(x)&=\frac{(-1)^{n}2^{2n+1}e^{x^{2}}}{\sqrt{\pi }}
 \int _{0}^{\infty }ds\,s^{2n}e^{-s^{2}}\cos (2xs), \quad x\in \R , \\
 H_{2n+1}(x)&=\frac{(-1)^{n}2^{2n+2}e^{x^{2}}}{\sqrt{\pi }}
 \int _{0}^{\infty }ds\,s^{2n+1}e^{-s^{2}}\sin (2xs), \quad x\in \R , 
\end{align*}
where $n$ is any nonnegative integer. 
Dealing with other values of $\mu $ as well, we put the above-mentioned 
reduction in \pref{;preduc} below, which recovers in part 
Theorem~4.2 of \cite{duf} by Dufresne. 
For every $\mu \in \R $, we denote by $H_{\mu }$ the Hermite 
function of degree $\mu $ and recall its integral representation 
when $\mu >-1$: 
\begin{align}\label{;herf}
 H_{\mu }(x)=\frac{2^{\mu +1}e^{x^{2}}}{\sqrt{\pi }}
 \int _{0}^{\infty }ds\,s^{\mu }e^{-s^{2}}
 \cos \left( 
 2xs-\frac{\pi \mu }{2}
 \right) ,\quad x\in \R 
\end{align}
(see Section~10.2 and equation~\thetag{10.5.5} in \cite{leb} for 
the definition of the Hermite functions and the integral representation \eqref{;herf}, respectively). 

\begin{prop}\label{;preduc}
Let $\mu >-1$. For every $t>0$, the law of $\da{\mu }_{t}$ admits the density 
function expressed by 
\begin{align}
  &\frac{\pr (\da{\mu }_{t}\in dv)}{dv} \notag \\
  &=\frac{C_{\mu }(t)}{\sqrt{v^{3-\mu }}}
  \ex \!\left[ 
  \exp \left( 
  -\frac{\cosh ^{2}B_{t}}{2v}
  \right) 
  H_{\mu }\!\left( 
  \frac{\sinh B_{t}}{\sqrt{2v }}
  \right) \cosh B_{t}\cos \left\{ 
  \frac{\pi }{2}\left( \frac{B_{t}}{t}-\mu \right) 
  \right\} 
  \right] \label{;Amlaw}
\end{align}
for $v>0$, where 
$
C_{\mu }(t)=\bigl( 1/\sqrt{2^{\mu +1}\pi }\bigr) 
e^{\pi ^{2}/(8t)-\mu ^{2}t/2}
$. 
\end{prop}

Recall that 
\begin{align*}
 H_{0}(x)=1, \quad H_{1}(x)=2x, \quad x\in \R 
\end{align*}
(see, e.g., \cite[p.~60]{leb}). Expression \eqref{;Amlaw} 
in the case $\mu =0$ was obtained by several 
authors, for which we refer the reader to \cite[p.~223]{duf} as well as 
the beginning of \cite[Subsection~2.2]{my2003} (see also 
\rref{;rdiff}\thetag{4} above); an alternative derivation 
of \eqref{;Amlaw} in the case $\mu =0$ and $\mu =1$ 
without relying on \tref{;mt}, will be found in \sref{;sfd} (see 
equations~\eqref{;A0law} and \eqref{;A1law} therein). 
If we consider the law of $1/(2\da{\mu }_{t})$, 
then from \eqref{;Amlaw}, we partly recover formula~\thetag{4.9} in 
\cite[Theorem~4.2]{duf} due to Dufresne, who also shows that the formula is 
valid for $\mu \le -1$ as well, by developing a recurrence relation 
that connects the law of $1/(2\da{\nu }_{t})$ with that of 
$1/(2\da{\mu }_{t})$ when $\nu <\mu $. We do not pursue it here with  
generality, however, if we repeat integration by 
parts as necessary appealing to \eqref{;qpderi1}, 
then \tref{;mt} enables us to reduce the computation of the 
case $\mu \le -1$ to a situation where formula \eqref{;herf} 
applies or the function $H_{-1}$ emerges; see \rref{;ribp}. 
We also note that in \cite{sch}, a contour integral representation 
for the density of the law of $\da{\mu }_{t}$, is given 
in terms of Hermite functions and compared with 
Dufresne's representation. 

We give an outline of the paper. 
In \sref{;prfmt}, we prove \tref{;mt}; we do this by 
preparing \lref{;lequiv} which shows the equivalence of certain 
three relations for expectations relative to Brownian motion. 
Since in the proof of \tref{;mt}, one implication between 
two of the three relations is used, we give its proof in 
\sref{;prfmt} and the rest of the proof of the lemma is 
provided in the appendix. \psref{;pderi} and \ref{;preduc} are 
proven in \sref{;prfprops}. We prove in \sref{;sfd} a family of integral 
identities that embraces relations in \tref{;mt}, which is then 
applied to the derivation of other integral representations 
of $\Theta $ relevant to \tref{;mt}. 
Finally, in the appendix, we complete the proof of \lref{;lequiv}. 

\section{Proof of \tref{;mt}}\label{;prfmt}
From now on, we fix $t>0$. This section is devoted to the proof of \tref{;mt}. 
Let two real-valued functions $\f $ and $\g $ on $\R $ be continuous 
for simplicity and suppose that they are even functions and satisfy 
\begin{align}\label{;intcond}
 \ex \!\left[ 
 |\f (B_{t})|
 \right] <\infty  && \text{and} &&  
 \ex \!\left[ 
 |\g (B_{t})|\cosh B_{t}
 \right] <\infty . 
\end{align}

\begin{lem}\label{;lequiv}
Under condition \eqref{;intcond}, the following three relations 
are equivalent: 
\begin{align}
  \thetag{i}\ &\frac{2}{\pi }\ex \!\left[ 
  \frac{\f (B_{t})\cosh B_{t}}{\cosh (2B_{t})+\cosh (2x)}
  \right] 
  =\frac{1}{\sqrt{2\pi t}}\exp \left( 
  -\frac{x^{2}}{2t}
  \right) \g (x) \quad \text{for any }x\in \R ; \label{;rel3}\\
  \thetag{ii}\ &\ex \!\left[ 
  e^{-r\cosh B_{t}}\f (B_{t})
  \right] 
  =\ex \!\left[ 
  \g (B_{t})\cosh B_{t}\cos (r\sinh B_{t})
  \right] \quad \text{for any }r\ge 0; \label{;rel2}\\
  \thetag{iii}\ &\ex \!\left[ 
  \frac{\f (B_{t})}{\cosh B_{t}+\cosh x}
  \right] 
  =\ex \!\left[ 
  \frac{\g (B_{t})}{\cosh (x+B_{t})}
  \right] \quad \text{for any }x\in \R .  \label{;rel1}
\end{align}
\end{lem}

Since, in the proof of \tref{;mt}, we only use the implication 
from \thetag{i} to \thetag{ii}, we give a proof of it 
below and postpone proofs of other implications to 
the appendix. 

\begin{proof}[Proof of \thetag{i} $\Rightarrow $ \thetag{ii} in \lref{;lequiv}]
For every $r\ge 0$, we integrate both sides of \eqref{;rel3} multiplied by 
$\cosh x\cos (r\sinh x)$ with respect to $x\in \R$. 
Then the right-hand side turns into that of \eqref{;rel2} in view of the 
latter condition in \eqref{;intcond}. On the other hand, provided that we are 
allowed to use Fubini's theorem, the left-hand side turns into 
 \begin{align*}
 \frac{2}{\pi }\ex \!\left[ 
 F(B_{t})\cosh B_{t}\int _{\R }dx\,
 \frac{\cosh x\cos (r\sinh x)}{\cosh (2B_{t})+\cosh (2x)}
 \right] =\ex \!\left[ 
 e^{-r\cosh B_{t}}\f (B_{t})
 \right] 
\end{align*}
as claimed. Here for the equality, we used the fact that 
for any $b\in \R $ and $r\ge 0$, 
\begin{align}\label{;fact1}
 \int _{\R }dx\,
 \frac{\cosh x\cos (r\sinh x)}{\cosh (2b)+\cosh (2x)}
 =\frac{\pi }{2\cosh b}e^{-r\cosh b}, 
\end{align}
which follows readily by noting 
$
 \cosh (2b)+\cosh (2x)
 =2\left( \cosh ^{2}b+\sinh ^{2}x\right) 
$, 
and changing the variables 
with $\sinh x=y\cosh b,\,y\in \R $. Usage of Fubini's theorem 
mentioned above is justified by taking $r=0$ in \eqref{;fact1}; 
indeed, 
\begin{align}
 \frac{2}{\pi }\int _{\R }dx\,
 \cosh x\,\ex \!\left[ 
 \frac{|F(B_{t})|\cosh B_{t}}{\cosh (2B_{t})+\cosh (2x)}
 \right] =\ex \!\left[ 
 |F(B_{t})|
 \right] , \label{;fubi1}
\end{align}
which is assumed to be finite in \eqref{;intcond}. The 
proof is complete. 
\end{proof}

In what follows, we denote by $\C $ the complex plane and write 
$i=\sqrt{-1}$. A pair of functions $F$ and $G$ fulfilling relation 
\eqref{;rel3} may be obtained by the residue theorem applied to a 
meromorphic function $f$ of the form 
\begin{align*}
 f(z)=\frac{J(z)}{\cosh (2z)+\cosh (2x)}\exp \left( 
 -\frac{z^{2}}{2t}
 \right) ,\quad z\in \C , 
\end{align*}
where $x\in \R $ and $J(z),\,z\in \C$, is an odd entire function 
which will be taken to be either $\sinh (2z)$ or $\sinh z$ below. 
When $x\neq 0$, the poles $w$ of $f$ each of whose imaginary 
part $\mathrm{Im}\,w$ lies between $0$ and $\pi $, are two points 
$\pm x+ (\pi /2)i$. By taking a rectangular contour circling 
these poles and having its two sides on the two lines 
$\mathrm{Im}\,z=0$ and $\mathrm{Im}\,z=\pi $, 
residue calculus yields, at least heuristically, 
\begin{align*}
 &\frac{1}{2\pi i}
 \int _{\R }\frac{d\xi }{\cosh (2\xi )+\cosh (2x)}
 \exp \left( 
 -\frac{\xi ^{2}}{2t}+\frac{\pi ^{2}}{2t}
 \right) J(\xi +\pi i)\exp \left( 
 -\frac{\pi \xi }{t}i
 \right) \\
 &=\frac{1}{2\sinh (2x)}\exp \left( 
 -\frac{x^{2}}{2t}+\frac{\pi ^{2}}{8t}
 \right) \left\{ 
 J(x+\pi i/2)\exp \left( 
 -\frac{\pi x}{2t}i
 \right) +J(x-\pi i/2)\exp \left( 
 \frac{\pi x}{2t}i
 \right) 
 \right\} 
\end{align*}
for $x\neq 0$. When $J(z)=\sinh (2z)$ and $\sinh z$, the above 
computation is justified, yielding the following lemma: 

\begin{lem}\label{;liden}
 It holds that for any $x\in \R $, 
 \begin{align}
  \frac{1}{\pi }\ex \!\left[ 
  \frac{
  \sinh B_{t}\cosh B_{t}\sin \left( \pi B_{t}/t\right) 
  }
  {
  \cosh (2B_{t})+\cosh (2x)
  }
  \right] 
  &=\frac{1}{\sqrt{2\pi t}}
  \exp \left( 
  -\frac{x^{2}}{2t}-\frac{3\pi ^{2}}{8t}
  \right) \cos \left( 
  \frac{\pi x}{2t}
  \right) , \label{;iden1}\\
  \frac{1}{\pi }\ex \!\left[ 
  \frac{
  \sinh B_{t}\sin \left( \pi B_{t}/t\right) 
  }
  {
  \cosh (2B_{t})+\cosh (2x)
  }
  \right] 
  &=\frac{1}{\sqrt{2\pi t}}
  \exp \left( 
  -\frac{x^{2}}{2t}-\frac{3\pi ^{2}}{8t}
  \right) S(x), \label{;iden2}
 \end{align}
 where in the latter identity, the function $S(x)\equiv S(x,t),\,x\in \R $, 
 is given by 
 \begin{align*}
  S(x)=
  \begin{cases}
   \sin (\frac{\pi }{2t}x)/\sinh x & \text{for $x\neq 0$}, \\
   \pi /(2t) & \text{for $x=0$}. 
  \end{cases}
 \end{align*}
\end{lem}

Validity of \eqref{;iden1} and \eqref{;iden2} at $x=0$ follows by passing 
to the limit as $x\to 0$. We remark that these two identities are found 
in Lemmas~3.1 and 3.2 of \cite{my2003} by Matsumoto--Yor, in which 
paper those two lemmas are used to show that, in the case 
$\mu =0$ and $1$, expression \eqref{;Amlaw0} 
with Yor's formula \eqref{;irepr0} inserted in coincides with 
\eqref{;Amlaw}. 
What is revealed in the present paper is that coincidence for any 
$\mu >-1$ is also reduced to the above two identities \eqref{;iden1} 
and \eqref{;iden2} via \tref{;mt}. 

We are in a position to prove the theorem. 

\begin{proof}[Proof of \tref{;mt}]
 First we prove \eqref{;irepr1}. Identity \eqref{;iden1} tells us that 
 we may take 
 \begin{align*}
  \f (x)=\frac{1}{2}\sinh x\sin \left( 
  \frac{\pi x}{t}
  \right)  && \text{and} && 
  \g (x)=\exp \left( -\frac{3\pi ^{2}}{8t}\right) 
  \cos \left( 
  \frac{\pi x}{2t}
  \right) 
 \end{align*}
 in relation \eqref{;rel3}. It is clear that these functions 
 fulfill the integrability condition \eqref{;intcond}. Therefore 
 by \lref{;lequiv}, we have for every $r\ge 0$, 
 \begin{align*}
  \frac{1}{2}\ex \!\left[ 
  e^{-r\cosh B_{t}}\sinh B_{t}\sin \left( 
  \frac{\pi }{t}B_{t}
  \right) 
  \right] 
  =\exp \left( -\frac{3\pi ^{2}}{8t}\right) \ex \!\left[ 
  \cosh B_{t}\cos (r\sinh B_{t})\cos \left( 
  \frac{\pi }{2t}B_{t}
  \right) 
  \right] . 
 \end{align*}
 Now representation \eqref{;irepr1} follows from this and Yor's 
 formula \eqref{;irepr0d}. 

 We proceed to the proof of \eqref{;irepr2}. The second identity 
 \eqref{;iden2} in \lref{;liden} shows that we may take 
 in \eqref{;rel3} 
 \begin{align*}
  \f (x)=\frac{1}{2}\frac{\sinh x}{\cosh x}\sin \left( 
  \frac{\pi x}{t}
  \right) && \text{and} && 
  \g (x)=\exp \left( -\frac{3\pi ^{2}}{8t}\right) 
  S(x), 
 \end{align*}
 which pair also fulfills \eqref{;intcond}. Therefore by 
 \lref{;lequiv}, we have for every $r\ge 0$, 
 \begin{align*}
  \frac{1}{2}\ex \!\left[ 
  e^{-r\cosh B_{t}}\frac{\sinh B_{t}}{\cosh B_{t}}\sin \left( 
  \frac{\pi }{t}B_{t}
  \right) 
  \right] 
  =\exp \left( -\frac{3\pi ^{2}}{8t}\right) \ex \!\left[ 
  \cosh B_{t}\cos (r\sinh B_{t})S(B_{t})
  \right] . 
 \end{align*}
 Thanks to the fact that 
 \begin{align}\label{;scfinite}
  \ex \!\left[ |\sinh B_{t}|\right] 
  \le \ex \!\left[ \cosh B_{t}\right] <\infty , 
 \end{align}
 differentiating both sides of the last identity with respect to $r$ and  
 appealing to formula \eqref{;irepr0d} again, we arrive at \eqref{;irepr2}. 
 
 Representation \eqref{;irepr3} is a consequence of summation of 
 \eqref{;irepr1} and \eqref{;irepr2}. To prove \eqref{;irepr4}, 
 fix $\nu \in \R $. Using the addition theorem, we develop 
 \begin{align*}
  &\cos (r\sinh B_{t}-\nu )\cos \left( 
  \frac{\pi }{2t}B_{t}-\nu 
  \right) \\
  &=\cos (r\sinh B_{t})\cos \left( 
  \frac{\pi }{2t}B_{t}
  \right) \cos ^{2}\nu 
  +\sin (r\sinh B_{t})\sin \left( 
  \frac{\pi }{2t}B_{t}
  \right) \sin ^{2}\nu \\
  &\quad +R(B_{t})\sin \nu \cos \nu 
 \end{align*}
 with $R(x),\,x\in \R $, an odd function such that 
 $
 \ex \left[ 
 |R(B_{t})|\cosh B_{t}
 \right] <\infty 
 $. 
 Hence the right-hand side 
 of the claimed identity \eqref{;irepr4} is equal to 
 \begin{align*}
  &\frac{r}{\pi }\exp \left( 
  \frac{\pi ^{2}}{8t}
  \right) \Bigl\{ 
  \ex \!\left[ 
  \cosh B_{t}\cos (r\sinh B_{t})\cos \left( 
  \frac{\pi }{2t}B_{t}
  \right) 
  \right] \cos ^{2}\nu \\
  &\qquad \qquad \qquad +
  \ex \!\left[ 
  \cosh B_{t}\sin (r\sinh B_{t})\sin \left( 
  \frac{\pi }{2t}B_{t}
  \right) 
  \right] \sin ^{2}\nu 
  \Bigr\} \\
  &=\Theta (r,t)\!\left( \cos ^{2}\nu +\sin ^{2}\nu \right) 
 \end{align*}
 by \eqref{;irepr1} and \eqref{;irepr2}, which shows \eqref{;irepr4} 
 and completes the proof of the theorem. 
\end{proof}

\section{Proofs of \psref{;pderi} and \ref{;preduc}}\label{;prfprops} 

In the sequel, we set the function $g(r)\equiv g(r,t),\,r>0$, by 
\begin{align*}
 g(r):=\ex \!\left[ 
  \cosh B_{t}\cos (r\sinh B_{t})\cos \left( 
  \frac{\pi }{2t}B_{t}
  \right) \right] . 
\end{align*}
As seen in the previous section, it holds that for any $r>0$, 
\begin{align}
 &g(r)=\ex \!\left[ 
 \cosh B_{t}\sin (r\sinh B_{t})\sin \left( 
 \frac{\pi }{2t}B_{t}
 \right) \right] \notag  
\intertext{and} 
 &\qquad \ \ \quad \Theta (r,t)=\frac{r}{\pi }\exp \left( 
 \frac{\pi ^{2}}{8t}
 \right) g(r) . \label{;Theta}
\end{align}

\subsection{Proof of \pref{;pderi}}\label{;prfpderi}

In this subsection, we prove \pref{;pderi}. 

\begin{proof}[Proof of \pref{;pderi}]
 By relation \eqref{;Theta}, it suffices to show that 
 \begin{align}\label{;deri0d}
  \lim _{r\to 0+}g^{(n)}(r)=0,\quad n=0,1,2,\ldots . 
 \end{align}
 By observing the fact that 
 \begin{align}\label{;finite}
  \ex \!\left[ 
  \cosh B_{t}\left| \sinh B_{t}\right| ^{n}
  \right] <\infty 
 \end{align}
 for any nonnegative integer $n$, successive differentiation 
 of the above two representations of $g$ yields 
 \begin{align*}
  g^{(2n)}(r)&=(-1)^{n}
  \ex \!\left[ 
  \cosh B_{t}\sinh ^{2n}\!B_{t}\sin (r\sinh B_{t})\sin \left( 
  \frac{\pi }{2t}B_{t}
  \right) \right] , \\
  g^{(2n+1)}(r)&=(-1)^{n+1}
  \ex \!\left[ 
  \cosh B_{t}\sinh ^{2n+1}\!B_{t}\sin (r\sinh B_{t})\cos \left( 
  \frac{\pi }{2t}B_{t}
  \right) \right] 
 \end{align*}
 for every nonnegative integer $n$, from which \eqref{;deri0d} 
 follows readily by the dominated convergence theorem. 
\end{proof}

\begin{rem}\label{;rcontr}
 Another direct way of convincing ourselves of \eqref{;deri0d} is to 
 consider the mapping 
 \begin{align}\label{;map}
  \R \ni r\mapsto 
  \ex \!\left[ 
  \cosh B_{t}\cos \left( 
  \frac{\pi }{2t}B_{t}-r\sinh B_{t}
  \right) 
  \right] , 
 \end{align}
 which determines a $C^{\infty }$-function 
 thanks to \eqref{;finite}, agrees with $2g(r)$ for $r>0$, 
 and is identically zero on $(-\infty ,0)$ in view of \eqref{;diff0}, 
 and hence would yield a contradiction if \eqref{;deri0d} were not 
 the case. A fact of independent interest following from the 
 above argument is that the function given by \eqref{;map} 
 provides an example of a $C^{\infty }$-function on $\R $ that is 
 not analytic at $r=0$;  in fact,  
 \begin{align*}
  \ex \!\left[ 
  \cosh B_{t}\cos \left( 
  \frac{\pi }{2t}B_{t}-r\sinh B_{t}
  \right) 
  \right] >0 
 \end{align*}
 whenever $r>0$, since the left-hand side agrees with 
 \begin{align}\label{;rmap}
  \sqrt{\frac{2\pi }{t}}\exp \left( r-\frac{\pi ^{2}}{8t}\right) 
  \frac{\displaystyle P\!\left( 1/A_{t}\in dr \mid B_{t}=0\right) }{dr}
 \end{align}
 and, under the pinned measure $P(\,\cdot \mid B_{t}=0)$, 
 $\mathrm{ess\,sup}\,A_{t}=\infty $ and $\mathrm{ess\,inf}\,A_{t}=0$. 
 Expression \eqref{;rmap} is a consequence of \eqref{;jl1} 
 together with \eqref{;irepr3}.
\end{rem}

\subsection{Proof of \pref{;preduc}}\label{;prfpreduc}

In this subsection, we prove \pref{;preduc}. 

\begin{proof}[Proof of \pref{;preduc}]
 We insert representation \eqref{;irepr4} into \eqref{;Amlaw0} 
 putting $\nu =\pi \mu /2$. Then for $\mu >-1$, 
 Fubini's theorem entails that \eqref{;Amlaw0} is rewritten as 
 \begin{align}\label{;fubini}
  \frac{1}{\pi }\exp \left( 
  \frac{\pi ^{2}}{8t}-\frac{\mu ^{2}}{2}t
  \right) 
  v^{\mu -1}\exp \left( 
  -\frac{1}{2v}
  \right) \ex \!\left[ 
  h(B_{t})\cosh B_{t}\cos \left\{ 
  \frac{\pi }{2}\left( 
  \frac{B_{t}}{t}-\mu 
  \right) 
  \right\} 
  \right] , 
 \end{align}
 where we set the function $h(x),\,x\in \R $, by 
 \begin{align*}
  h(x)=\int _{0}^{\infty }dr\,r^{\mu }\exp \left( 
  -\frac{v}{2}r^{2}
  \right) \cos \left( 
  r\sinh x-\frac{\pi \mu }{2}
  \right) , 
 \end{align*}
 which is equal, by changing the variables with 
 $r=\sqrt{(2/v)}s$, to 
 \begin{align*}
  &\left( 
  \frac{2}{v}
  \right) ^{(\mu +1)/2}
  \int _{0}^{\infty }ds\,s^{\mu }e^{-s^{2}}
  \cos \left( 
  2\frac{\sinh x}{\sqrt{2v}}s-\frac{\pi \mu }{2}
  \right) \\
  &=\sqrt{\frac{\pi }{(2v)^{\mu +1}}}
  \exp \left( 
  -\frac{\sinh ^{2}x}{2v}
  \right) 
  H_{\mu }\!\left( 
  \frac{\sinh x}{\sqrt{2v}}
  \right) 
 \end{align*}
 by the integral formula \eqref{;herf} of $H_{\mu }$ 
 with $\mu >-1$. Inserting the last expression of $h$ into 
 \eqref{;fubini} and rearranging terms lead to \eqref{;Amlaw} and 
 conclude the proof. 
\end{proof}

We end this section with a remark on the case $\mu \le -1$. 

\begin{rem}\label{;ribp}
 We take $\mu =-3/2$ and $-2$ as an illustration. Noting relation 
 \eqref{;Theta}, we use the function 
 $g$ to rewrite the integral in \eqref{;Amlaw0} with respect to $r$ as 
 \begin{align*}
  \frac{1}{\pi }\exp \left( 
  \frac{\pi ^{2}}{8t}
  \right) I(\mu ) && \text{with} && 
  I(\mu ):=\int _{0}^{\infty }dr\,r^{\mu }\exp \left( 
  -\frac{v}{2}r^{2}
  \right) g(r). 
 \end{align*}
 The proof of \pref{;preduc} shows that $I(\mu )$ may be expressed 
 in terms of $H_{\mu }$ when $\mu >-1$. 
 If we take $\mu =-3/2$, then integration by parts yields 
 \begin{align}\label{;qribp1}
  I(-3/2)=-2vI(1/2)+2\int _{0}^{\infty }\frac{dr}{r^{1/2}}\,\exp \left( 
  -\frac{v}{2}r^{2}
  \right) g'(r)  
 \end{align}
 owing to \eqref{;qpderi1}. 
 Recalling \eqref{;irepr4}, we have 
 \begin{align*}
  g'(r)&=-\ex \!\left[ 
  \cosh B_{t}\sinh B_{t}\sin (r\sinh B_{t}-\nu )\cos \left( 
  \frac{\pi }{2t}B_{t}-\nu 
  \right) 
  \right] \\
  &=-\ex \!\left[ 
  \cosh B_{t}\sinh B_{t}\cos \left( r\sinh B_{t}-\nu -\frac{\pi }{2}\right) 
  \cos \left( 
  \frac{\pi }{2t}B_{t}-\nu 
  \right) 
  \right] 
 \end{align*}
 for every $r>0$ and $\nu \in \R $. Therefore choosing 
 $\nu =-(3/4)\pi $ and appealing to Fubini's theorem and \eqref{;herf}, 
 we may express the second term on the right-hand side of \eqref{;qribp1} 
 in terms of $H_{-1/2}$. In the case $\mu =-2$, we have 
 in the same way as above, 
 \begin{align}\label{;qribp2}
  I(-2)=-vI(0)+\int _{0}^{\infty }\frac{dr}{r}\,\exp \left( 
  -\frac{v}{2}r^{2}
  \right) g'(r). 
 \end{align}
 Noting that 
 \begin{align*}
  g'(r)=-\ex \left[ 
  \cosh B_{t}\sinh B_{t}\sin (r\sinh B_{t})\cos \left( 
  \frac{\pi }{2t}B_{t}
  \right) 
  \right] ,\quad r>0, 
 \end{align*}
 and that $|\sin (r\sinh B_{t})/r|\le |\sinh B_{t}|$ for any $r>0$, 
 Fubini's theorem entails that the second term on the right-hand side 
 of \eqref{;qribp2} is written as 
 \begin{align*}
  -\ex \!\left[ 
  \cosh B_{t}\sinh B_{t}\cos \left( \frac{\pi }{2t}B_{t}
  \right) \int _{0}^{\infty }\frac{dr}{r}\,\exp \left( 
  -\frac{v}{2}r^{2}
  \right) \sin (r\sinh B_{t})
  \right] . 
 \end{align*}
 By the formulae 
 \begin{align}
  \int _{0}^{\infty }\frac{ds}{s}\,e^{-s^{2}}\sin (2xs)
  &=\sqrt{\pi }\int _{0}^{x}dy\,e^{-y^{2}},\quad x\ge 0, \label{;erf}\\
  H_{-1}(x)&=e^{x^{2}}\int _{x}^{\infty }dy\,e^{-y^{2}},
  \quad x\in \R , \label{;herf-1}
 \end{align}
 the integrand in the last expectation may be expressed in terms of $H_{-1}$. 
 As for formula \eqref{;herf-1}, see \cite[equation~(10.5.3)]{leb}. 
 Formula \eqref{;erf} may easily be verified by writing 
 \begin{align*}
  \frac{\sin (2xs)}{s}=2\int _{0}^{x}dy\,\cos (2sy), \quad s>0, 
 \end{align*}
 and using Fubini's theorem. The above illustration in the 
 two cases indicates that, when $-(m+1)\le \mu <-m$ 
 for some positive integer $m$, the density \eqref{;Amlaw0} 
 of the law of $\da{\mu }_{t}$ admits an expression 
 in terms of $H_{\mu +m}$ and $H_{\mu +m+1}$ (while 
 in the case $\mu =-1$, only $H_{-1}$ emerges).
\end{rem}

\section{A further discussion}\label{;sfd}
In this section, we develop further the discussion used in 
deriving \lref{;liden} to obtain a family of integral identities 
that includes relations \eqref{;irepr3} and \eqref{;diff0} 
as its special cases, which, in turn, yields \eqref{;irepr1} and 
\eqref{;irepr2}; another set of integral representations of 
$\Theta $ relevant to \tref{;mt}, is also provided.  

We keep $t>0$ fixed. We prove 

\begin{prop}\label{;punif}
 For every $r,\la \ge 0$ and $\ga \in \R $, it holds that 
 \begin{equation}\label{;equnif}
  \begin{split}
   &\ex \!\left[ 
   e^{-r\cosh B_{t}}\cos \left( 
   \frac{\pi }{2t}B_{t}+\la \sinh B_{t}-\ga B_{t}
   \right) 
   \right] \\
   &=\exp \left( 
   \frac{\pi \ga }{2}-\frac{\pi ^{2}}{8t}
   \right) 
   \ex \!\left[ 
   e^{-\la \cosh B_{t}}
   \cos \left( 
   r\sinh B_{t}+\ga B_{t}
   \right) 
   \right] . 
  \end{split}
 \end{equation}
\end{prop}

Before giving a proof of the above proposition, 
we explain in the remark below how to obtain 
several relations in \sref{;intro} from \eqref{;equnif}: 
note that, thanks to \eqref{;scfinite}, 
differentiating both sides of \eqref{;equnif} 
at $\la =0$ yields the relation that for every $r\ge 0$ and 
$\ga \in \R $, 
\begin{equation}\label{;equnifd}
 \begin{split}
   &\ex \!\left[ 
   e^{-r\cosh B_{t}}\sinh B_{t}\sin \left\{ 
   \left( \frac{\pi }{2t}-\ga \right) B_{t}
   \right\}  
   \right] \\
   &=\exp \left( 
   \frac{\pi \ga }{2}-\frac{\pi ^{2}}{8t}
   \right) 
   \ex \!\left[ 
   \cosh B_{t}
   \cos \left( 
   r\sinh B_{t}+\ga B_{t}
   \right) 
   \right] . 
 \end{split}
\end{equation}

\begin{rem}\label{;rrecover}
\thetag{1} 
Taking $\ga =-\pi /(2t)$ in \eqref{;equnifd}, we have 
relation \eqref{;irepr3}. If we take $\ga =\pi /(2t)$, then 
relation \eqref{;diff0} also follows, which, together with 
\eqref{;irepr3}, entails \eqref{;irepr1} and \eqref{;irepr2}.  

\noindent 
\thetag{2} Moreover, if we take $\ga =0$ in \eqref{;equnifd}, then 
in view of the implication from \thetag{ii} to \thetag{i} 
in \lref{;lequiv}, we have 
\begin{align}\label{;iden3}
 \frac{1}{\pi }\exp \left( 
 \frac{\pi ^{2}}{8t}
 \right) 
 \ex \!\left[ 
 \frac{
 \sinh (2B_{t})\sin \left( \frac{\pi }{2t}B_{t}\right) 
 }
 {
 \cosh (2B_{t})+\cosh (2x)
 }
 \right] 
 =\frac{1}{\sqrt{2\pi t}}\exp \left( 
 -\frac{x^{2}}{2t}
 \right) \quad \text{for any $x\in \R $}
\end{align}
(see also \rref{;riden3} below), 
which may be restated, by replacing $x$ and $t$ by $x/2$ and 
$t/4$, respectively, and by using the scaling property of Brownian 
motion, as 
\begin{align*}
 \frac{1}{2\pi }\exp \left( 
 \frac{\pi ^{2}}{2t}
 \right) 
 \ex \!\left[ 
 \frac{
 \sinh B_{t}\sin \left( \frac{\pi }{t}B_{t}\right) 
 }
 {
 \cosh B_{t}+\cosh x
 }
 \right] 
 =\frac{1}{\sqrt{2\pi t}}\exp \left( 
 -\frac{x^{2}}{2t}
 \right) .  
\end{align*}
Note that the left-hand side is rewritten as 
\begin{align*}
 \frac{1}{2\pi }\exp \left( 
 \frac{\pi ^{2}}{2t}
 \right) 
 \int _{0}^{\infty }dr\,e^{-r\cosh x}\,
 \ex \!\left[ 
 e^{-r\cosh B_{t}}\sinh B_{t}\sin \left( 
 \frac{\pi }{t}B_{t}
 \right) 
 \right] 
\end{align*}
by Fubini's theorem. Therefore once the characterization 
\eqref{;LTr} of $\Theta $, namely the Laplace transform 
of $\Theta (r,t)/r$ in the variable $r>0$, is at our disposal, 
the integral representation \eqref{;irepr0} follows by 
the injectivity of Laplace transform. We also note that in view of 
\lref{;lequiv}, relation \eqref{;iden3} is equivalent to 
\begin{align*}
 \exp \left( 
 \frac{\pi ^{2}}{8t}
 \right) 
 \ex \!\left[ 
 \frac{
 \sinh B_{t}\sin \left( \frac{\pi }{2t}B_{t}\right) 
 }
 {
 \cosh B_{t}+\cosh x
 }
 \right] 
 =\ex \!\left[ 
 \frac{1}{\cosh (x+B_{t})}
 \right] \quad \text{for any }x\in \R , 
\end{align*}
which is a relation observed in \cite[Subsection~A.3]{har}. 
\end{rem}

It would also be of interest to note that by taking 
$\ga =\pi /(4t)$ in \eqref{;equnif}, there holds the 
following symmetry with respect to the variables 
$r,\la \ge 0$: 
\begin{align*}
 \ex \!\left[ 
 e^{-r\cosh B_{t}}\cos \left( 
 \la \sinh B_{t}+\frac{\pi }{4t}B_{t}
 \right) 
 \right] 
 =\ex \!\left[ 
 e^{-\la \cosh B_{t}}\cos \left( 
 r \sinh B_{t}+\frac{\pi }{4t}B_{t}
 \right) 
 \right] . 
\end{align*}

In order to prove \pref{;punif}, we consider a contour 
integral of a meromorphic function $f(z),\,z\in \C $, 
of the form 
\begin{align*}
 f(z)=\frac{\sinh (2z)}{\cosh (2z)+\cosh (2x)}
 \exp \left\{ 
 -\frac{1}{2t}\left( 
 z-\frac{\pi }{2}i
 \right) ^{2}
 \right\} \g \left( 
 z-\frac{\pi }{2}i
 \right) , 
\end{align*}
where $x\in \R $ is fixed and $\g $ denotes an 
even entire function. If we take the same contour 
as used in \sref{;prfmt} supposing $x\neq 0$, then we have 
for a suitable choice of $\g $, 
\begin{equation}\label{;rc}
 \begin{split}
 &\frac{1}{2\pi i}
 \int _{\R }d\xi \,\frac{\sinh (2\xi )}{\cosh (2\xi )+\cosh (2x)}
 \exp \left( 
 -\frac{\xi ^{2}}{2t}+\frac{\pi ^{2}}{8t}
 \right) \\
 &\qquad \times \left\{ 
 \exp \left( 
 \frac{\pi \xi }{2t}i
 \right) \g \left( 
 \xi -\frac{\pi }{2}i
 \right) 
 -\exp \left( 
 -\frac{\pi \xi }{2t}i
 \right) \g \left( 
 \xi +\frac{\pi }{2}i
 \right) 
 \right\} \\
 &=\exp \left( 
 -\frac{x^{2}}{2t}
 \right) \g (x)
 \end{split}
\end{equation}
by residue calculus applied to $f$. As will be seen, we are 
allowed to take $\g (z)=\exp (-\la \cosh z)\cos (\ga z)$ and 
$\exp (-\la \cosh z)\sin (\ga z)\sinh z$ for $\la \ge 0$ and 
$\ga \in \R $, obtaining the following lemma: set 
\begin{align*}
 \Sigma (x,\ga )
 \equiv \Sigma _{\la }(x,\ga )
 :=e^{\pi \ga /2}
 \sin \left( 
 \frac{\pi x}{2t}+\la \sinh x+\ga x
 \right) ,\quad x\in \R . 
\end{align*}

\begin{lem}\label{;lrc}
 For any $\la \ge 0$ and $\ga \in \R $, relation \eqref{;rel3} 
 holds for the following pairs of functions $\f (x)$ and $\g (x)$, 
 $x\in \R $: 
 \begin{align*}
  \thetag{i}& \ 
  \begin{cases}
   \f (x)=\dfrac{1}{2}\exp \left( 
   \dfrac{\pi ^{2}}{8t}
   \right) \sinh x\left\{ 
   \Sigma (x,\ga )+\Sigma (x,-\ga )
   \right\} ,\\
   \g (x)=\exp (-\la \cosh x)\cos (\ga x);
  \end{cases}\\
  \thetag{ii}& \ 
  \begin{cases}
   \f (x)=-\dfrac{1}{2}\exp \left( 
   \dfrac{\pi ^{2}}{8t}
   \right) \sinh x\cosh x\left\{ 
   \Sigma (x,\ga )-\Sigma (x,-\ga )
   \right\} ,\\
   \g (x)=\exp (-\la \cosh x)\sin (\ga x)\sinh x. 
  \end{cases}
 \end{align*}
\end{lem}

\begin{proof}
 In the case $\g (z)=\exp (-\la \cosh z)\cos (\ga z),\,z\in \C $, 
 because of the fact that 
 \begin{align*}
  \left| 
  \exp \left\{ 
  -\la \cosh \left( 
  L+i\eta -\frac{\pi }{2}i
  \right) 
  \right\} 
  \right| =\exp (-\la \cosh L\sin \eta )\le 1
\end{align*} 
for any $L\in \R $ and $0\le \eta \le \pi $ when $\la \ge 0$, 
it follows readily that 
\begin{align*}
 \left| 
 \int _{0}^{\pi }d\eta \,f(L+i\eta )
 \right| \xrightarrow[|L|\to \infty ]{}0, 
\end{align*}
which justifies \eqref{;rc}, yielding pair~\thetag{i}. The same justification 
is also true in the case 
$\g (z)=\exp (-\la \cosh z)\sin (\ga z)\sinh z,\,z\in \C $, 
and leads to  \thetag{ii}. As it is clear that $\ga $ may be any 
complex number in the above argument, it is also possible to 
obtain \thetag{ii} from \thetag{i} by using the following 
relations: 
\begin{align*}
 \frac{\Sigma (x,\ga -i)-\Sigma (x,\ga +i)}{2i}&=-\Sigma (x,\ga )\cosh x, \\
 \frac{\cos \left\{ (\ga -i)x\right\} 
 -\cos \left\{ (\ga +i)x\right\} 
 }{2i}&=\sin (\ga x)\sinh x
\end{align*}
for any $x,\ga \in \R $. We conclude the proof of the lemma. 
\end{proof}

\begin{rem}\label{;riden3}
 Relation \eqref{;iden3} is nothing but the case 
 $\la =\ga =0$ in \thetag{i}; in other words, it is obtained 
 simply by taking $\g \equiv 1$ in \eqref{;rc}. 
\end{rem}

Using the above lemma, we prove \pref{;punif}. 

\begin{proof}[Proof of \pref{;punif}]
 By \lref{;lequiv} and \thetag{i} of \lref{;lrc}, we have for any 
 $r,\la \ge 0$ and $\ga \in \R $, 
 \begin{align*}
  &\frac{1}{2}\exp \left( 
  \frac{\pi ^{2}}{8t}
  \right) \ex\!\left[ 
  e^{-r\cosh B_{t}}
  \sinh B_{t}\left\{ 
  \Sigma _{\la }(B_{t},\ga )+\Sigma _{\la }(B_{t},-\ga )
  \right\} 
  \right] \\
  &=\ex \!\left[ 
  e^{-\la \cosh B_{t}}\cos (\ga B_{t})
  \cosh B_{t}\cos (r\sinh B_{t})
  \right] , 
 \end{align*}
 which shows that the first derivatives with respect to 
 $\la $ of the following two expressions agree: 
 \begin{align}
  &\frac{1}{2}\exp \left( 
  \frac{\pi ^{2}}{8t}
  \right) \ex\!\left[ 
  e^{-r\cosh B_{t}}
  \left\{ 
  \Tilde{\Sigma }_{\la }(B_{t},\ga )+\Tilde{\Sigma }_{\la }(B_{t},-\ga )
  \right\} 
  \right] , \label{;ftn1}\\
  &\ex \!\left[ 
  e^{-\la \cosh B_{t}}\cos (\ga B_{t})\cos (r\sinh B_{t})
  \right] , \label{;ftn2}
 \end{align}
where we set 
\begin{align*}
 \Tilde{\Sigma }_{\la }(x,\ga )
 :=e^{\pi \ga /2}\cos \left( 
 \frac{\pi x}{2t}+\la \sinh x+\ga x
 \right) ,\quad x,\ga \in \R . 
\end{align*}
(We are allowed to interchange the order of differentiation 
and expectation thanks to \eqref{;scfinite}.) 
Moreover, by the Riemann--Lebesgue lemma, 
the former expression \eqref{;ftn1} converges to $0$ as $\la \to \infty $, 
and by the bounded convergence theorem, the latter expression 
\eqref{;ftn2} does as well. Therefore the two expressions 
\eqref{;ftn1} and \eqref{;ftn2} agree. Similarly, 
by \lref{;lequiv} and \thetag{ii} of \lref{;lrc}, we have 
\begin{align*}
  &-\frac{1}{2}\exp \left( 
  \frac{\pi ^{2}}{8t}
  \right) \ex\!\left[ 
  e^{-r\cosh B_{t}}
  \sinh B_{t}\cosh B_{t}\left\{ 
  \Sigma _{\la }(B_{t},\ga )-\Sigma _{\la }(B_{t},-\ga )
  \right\} 
  \right] \\
  &=\ex \!\left[ 
  e^{-\la \cosh B_{t}}\sin (\ga B_{t})
  \sinh B_{t}\cosh B_{t}\cos (r\sinh B_{t})
  \right] . 
 \end{align*}
 If we consider the following two expressions  
 \begin{align}
  &\frac{1}{2}\exp \left( 
  \frac{\pi ^{2}}{8t}
  \right) \ex\!\left[ 
  e^{-r\cosh B_{t}}
  \left\{ 
  \Tilde{\Sigma }_{\la }(B_{t},\ga )-\Tilde{\Sigma }_{\la }(B_{t},-\ga )
  \right\} 
  \right] , \label{;ftn3}\\
  &\ex \!\left[ 
  e^{-\la \cosh B_{t}}\sin (\ga B_{t})\sin (r\sinh B_{t})
  \right] , \label{;ftn4}
 \end{align} 
 then by differentiating them with respect to $\la $ and $r$ 
 successively, and by using the same reasoning as above, 
 the last identity entails that those two expressions 
 also agree. Consequently, the difference of \eqref{;ftn1} 
 and \eqref{;ftn3} coincides with that of \eqref{;ftn2} and 
 \eqref{;ftn4}, which proves the proposition. 
\end{proof}

Identity \eqref{;equnifd} enables us to obtain yet another set of 
integral representations of $\Theta$, which 
we put in the next proposition. 

\begin{prop}\label{;pother}
 For every $r>0$ and $t>0$, it holds that 
 \begin{align}
  \Theta (r,t)&=-\frac{r}{\pi }\exp \left( \frac{9\pi ^{2}}{8t}\right) 
  \ex \!\left[ 
  \cosh B_{t}\cos \left( r\sinh B_{t}+\frac{\pi }{t}B_{t}\right) \cos \left( 
  \frac{\pi }{2t}B_{t}
  \right) 
  \right] \label{;other1}\\
  &=\frac{r}{\pi }\exp \left( \frac{9\pi ^{2}}{8t}\right) 
  \ex \!\left[ 
  \cosh B_{t}\sin \left( r\sinh B_{t}+\frac{\pi }{t}B_{t}\right) \sin \left( 
  \frac{\pi }{2t}B_{t}
  \right) 
  \right] \label{;other2}\\
  &=-\frac{r}{2\pi }\exp \left( \frac{9\pi ^{2}}{8t}\right) 
  \ex \!\left[ 
  \cosh B_{t}\cos \left( r\sinh B_{t}+
  \frac{3\pi }{2t}B_{t}
  \right) 
  \right] . \label{;other3}
 \end{align}
 More generally, we have for every $r>0$ and $t>0$, 
 \begin{align}\label{;other4}
  \Theta (r,t)&=\frac{r}{\pi }\exp \left( \frac{9\pi ^{2}}{8t}\right) 
  \ex \!\left[ 
  \cosh B_{t}\sin \left( r\sinh B_{t}+\frac{\pi }{t}B_{t}+\nu \right) 
  \sin \left( \frac{\pi }{2t}B_{t}-\nu 
  \right) 
  \right] , 
 \end{align}
 where $\nu \in \R $ is arbitrary. 
\end{prop}

\begin{proof}
 The third representation \eqref{;other3} follows by taking 
 $\ga =3\pi /(2t)$ in \eqref{;equnifd} and noting \eqref{;irepr0d}. 
 Then the first two representations \eqref{;other1} and 
 \eqref{;other2} are obtained by observing that their arithmetic mean 
 agrees with \eqref{;other3} and their difference 
 vanishes because of \eqref{;diff0}. The last representation 
 \eqref{;other4} is proven in the same way as in the proof of 
 \eqref{;irepr4}. 
\end{proof}

\begin{rem}\label{;rcomb}
By combining \tref{;mt} and \pref{;pother}, it is also possible to 
derive the following representations: for every 
$r>0$ and $t>0$, 
\begin{align*}
 \Theta (r,t)&=\frac{r}{2\pi }
 \frac{\exp \bigl( \frac{\pi ^{2}}{2t}\bigr) }
 {\sinh \!\left( \frac{\pi ^{2}}{2t}\right) }
 \exp \left( \frac{\pi ^{2}}{8t}\right) 
 \ex \!\left[ 
 \cosh B_{t}\cos \left( r\sinh B_{t}+\frac{\pi }{2t}B_{t}\right) \cos \left( 
 \frac{\pi }{t}B_{t}
 \right) 
 \right] \\
 &=\frac{r}{2\pi }
 \frac{\exp \bigl( \frac{\pi ^{2}}{2t}\bigr) }
 {\cosh \!\left( \frac{\pi ^{2}}{2t}\right) }
 \exp \left( \frac{\pi ^{2}}{8t}\right) 
 \ex \!\left[ 
 \cosh B_{t}\sin \left( r\sinh B_{t}+\frac{\pi }{2t}B_{t}\right) \sin \left( 
 \frac{\pi }{t}B_{t}
 \right) 
 \right] .
\end{align*}
In fact, replacing $\ga $ in \eqref{;equnifd} by $\pi /(2t)\pm \ga $, 
one may deduce that for every $r>0$, $t>0$ and $\ga \in \R $, 
\begin{align*}
 &\sinh \left ( \frac{\pi \ga }{2}\right) \ex \!\left[ 
 e^{-r\cosh B_{t}}\sinh B_{t}\sin (\ga B_{t})
 \right] \\
 &=\exp \left( \frac{\pi ^{2}}{8t}\right) 
 \ex \!\left[ 
 \cosh B_{t}\cos \left( r\sinh B_{t}+\frac{\pi }{2t}B_{t}\right) 
 \cos (\ga B_{t}) 
 \right], \\
 &\cosh \left ( \frac{\pi \ga }{2}\right) \ex \!\left[ 
 e^{-r\cosh B_{t}}\sinh B_{t}\sin (\ga B_{t})
 \right] \\
 &=\exp \left( \frac{\pi ^{2}}{8t}\right) 
 \ex \!\left[ 
 \cosh B_{t}\sin \left( r\sinh B_{t}+\frac{\pi }{2t}B_{t}\right) 
 \sin (\ga B_{t}) 
 \right] . 
\end{align*} 
\end{rem}

We recall from \cite[Proposition~3.3]{har} that for any $\la ,r\ge 0$, 
\begin{align}\label{;jlt}
 \ex\!\left[ 
 \exp \left( 
 -\la \eb{t}-\frac{\la ^{2}+r^{2}}{2}A_{t}
 \right) 
 \right] 
 =\ex\!\left[ 
 e^{-\la \cosh B_{t}}\cos (r\sinh B_{t})
 \right] . 
\end{align}
Hence, taking $\ga =0$ in \eqref{;equnif}, we have the following 
relation: 
\begin{align}\label{;jltd}
 \ex\!\left[ 
 \exp \left( 
 -\la \eb{t}-\frac{\la ^{2}+r^{2}}{2}A_{t}
 \right) 
 \right] 
 =\exp \left( 
 \frac{\pi ^{2}}{8t}
 \right) \ex\!\left[ 
 e^{-r\cosh B_{t}}\cos \left( \frac{\pi }{2t}B_{t}+\la \sinh B_{t}\right) 
 \right] . 
\end{align}
As for the former relation \eqref{;jlt}, we also refer to 
\cite[Proposition~2.4]{jw}, which may be regarded as the 
case where $r$ in \eqref{;jlt} is replaced by a purely imaginary 
number with modulus not exceeding $\la $.
To our knowledge, the latter relation \eqref{;jltd} has not been 
noticed before. We conclude this section by pointing out that 
one can easily derive from \eqref{;jltd} simple explicit expressions 
of the laws of $A_{t}$ and $A^{(1)}_{t}$; notice that in 
\cite{har}, relation \eqref{;jlt} is obtained independently of formula 
\eqref{;jl1}, by using so-called Bougerol's identity and a certain 
invariance formula for Cauchy random variable. 

Taking $\la =0$ in \eqref{;jltd}, we have 
\begin{align}\label{;A0lt}
 \ex \!\left[ 
 \exp \left( 
 -\frac{r^{2}}{2}A_{t}
 \right) 
 \right] 
 =\exp \left( 
 \frac{\pi ^{2}}{8t}
 \right) \ex \!\left[ 
 e^{-r\cosh B_{t}}\cos \left( 
 \frac{\pi }{2t}B_{t}
 \right) 
 \right] . 
\end{align}
Moreover, if we differentiate both sides of \eqref{;jltd} at $\la =0$, 
we also have 
\begin{align}\label{;A1lt}
 e^{t/2}\ex \!\left[ 
 \exp \left\{ 
 -\frac{r^{2}}{2}A^{(1)}_{t}
 \right\} 
 \right] 
 =\exp \left( 
 \frac{\pi ^{2}}{8t}
 \right) \ex \!\left[ 
 e^{-r\cosh B_{t}}\sinh B_{t}\sin \left( 
 \frac{\pi }{2t}B_{t}
 \right) 
 \right] , 
\end{align}
where on the left-hand side, we used the Cameron--Martin relation. 
Inserting the rewriting 
\begin{align*}
  e^{-r\cosh B_{t}}
  =\int _{0}^{\infty }dv\,\frac{\cosh B_{t}}{\sqrt{2\pi v^{3}}}
  \exp \left( 
  -\frac{\cosh ^{2}B_{t}}{2v}
  \right) \exp \left( 
  -\frac{r^{2}}{2}v
  \right) 
\end{align*}
into the right-hand sides of \eqref{;A0lt} and \eqref{;A1lt}, and 
using Fubini's theorem, we see that for $v>0$, 
\begin{align}
 \frac{\pr (A_{t}\in dv)}{dv}
 &=\exp \left( 
 \frac{\pi ^{2}}{8t}
 \right) 
 \ex \!\left[ 
 \frac{\cosh B_{t}}{\sqrt{2\pi v^{3}}}
 \exp \left( 
 -\frac{\cosh ^{2}B_{t}}{2v}
 \right) \cos \left( \frac{\pi }{2t}B_{t}\right) 
 \right] , \label{;A0law}\\
 \frac{\pr (\da{1}_{t}\in dv)}{dv}
 &=\exp \left( 
 \frac{\pi ^{2}}{8t}-\frac{t}{2}
 \right) 
 \ex \!\left[ 
 \frac{\sinh (2B_{t})}{\sqrt{2^{3}\pi v^{3}}}
 \exp \left( 
 -\frac{\cosh ^{2}B_{t}}{2v}
 \right) \sin \left( \frac{\pi }{2t}B_{t}\right) 
 \right] , \label{;A1law}
\end{align}
thanks to the injectivity of Laplace transform. These two expressions agree with \eqref{;Amlaw} when 
$\mu =0$ and $\mu =1$, respectively. 

\begin{rem}\label{;rboug}
 \thetag{1}~On the other hand, by taking $\la =0$ in the former relation 
 \eqref{;jlt}, we have 
 \begin{align*}
  \ex \!\left[ 
  \exp \left( 
  -\frac{r^{2}}{2}A_{t}
  \right) 
  \right] =\ex \!\left[ 
  \cos (r\sinh B_{t})
  \right] 
 \end{align*}
 for any $r\ge 0$, which relation is explained by Bougerol's original 
 identity 
 \begin{align*}
  \beta (A_{t})\stackrel{(d)}{=}\sinh B_{t}. 
 \end{align*}
 Here and below, $\{ \beta (s)\} _{s\ge 0}$ denotes a one-dimensional 
 standard Brownian motion independent of $B$. Moreover, if we differentiate 
 both sides of \eqref{;jlt} at $\la =0$, then by 
 the Cameron--Martin relation, we have 
 \begin{align*}
  \ex \!\left[ 
  \exp \left\{ 
  -\frac{r^{2}}{2}A^{(1)}_{t}
  \right\} 
  \right] 
  =\ex \!\left[ 
  \cos \left\{ 
  r\sinh (B_{t}+\ve t)
  \right\} 
  \right] 
 \end{align*}
 for any $r\ge 0$, where on the right-hand side, $\ve $ is 
 a Rademacher (or symmetric Bernoulli) random variable 
 taking values $\pm 1$ with probability $1/2$, independently of 
 $B$. By rewriting the left-hand side of the last identity as 
 \[
 \ex \!\left[ 
 \cos \left\{ 
 r\beta \bigl( A^{(1)}_{t}\bigr) 
 \right\} 
 \right] , 
 \] 
 the injectivity of Fourier transform entails the 
 following variant of Bougerol's identity: 
 \begin{align*}
  \beta \bigl( A^{(1)}_{t}\bigr) \stackrel{(d)}{=}\sinh (B_{t}+\ve t). 
 \end{align*}
 For Bougerol's identity and its variants including the above one, 
 see the survey \cite{vak} by Vakeroudis; different kinds of extensions 
 of Bougerol's identity may be found in \cite{har}. 
 
\noindent 
\thetag{2}~Relation~\eqref{;jltd} also enables us to derive the 
expression of the joint density \eqref{;jl1} with the integral 
representation \eqref{;irepr1} of $\Theta $ inserted in, but 
we omit details here. 
\end{rem}
\bigskip 

\noindent 
{\bf Acknowledgements.} The author would like to thank anonymous 
referees for their valuable comments.  

\appendix 
\section*{Appendix}\label{;app}
\renewcommand{\thesection}{A}
\setcounter{equation}{0}
\setcounter{prop}{0}
\setcounter{lem}{0}
\setcounter{rem}{0}

We complete the proof of \lref{;lequiv}. 

\begin{proof}[Proof of \thetag{ii}\ $\Rightarrow $\ \thetag{iii} in \lref{;lequiv}]
 Given $x\in \R $, we integrate both sides of \eqref{;rel2} multiplied by 
 $e^{-r\cosh x}$ with respect to $r\ge 0$. Then by Fubini's theorem, 
 the left-hand side turns into that of \eqref{;rel1}. Therefore it suffices to 
 show that 
 \begin{align}\label{;eqa1}
  \int _{0}^{\infty }dr\,e^{-r\cosh x}\,
  \ex \!\left[ 
  \g (B_{t})\cosh B_{t}\cos (r\sinh B_{t})
  \right] 
  =\ex \!\left[ 
  \frac{\g (B_{t})}{\cosh (x+B_{t})}
  \right] . 
 \end{align}
 By the latter condition in \eqref{;intcond}, we may also use Fubini's theorem 
 to rewrite the left-hand side of the claimed identity \eqref{;eqa1} as 
 \begin{align}
  &\ex \!\left[ 
  \g (B_{t})\cosh B_{t}
  \int _{0}^{\infty }dr\,e^{-r\cosh x}
  \cos (r\sinh B_{t})
  \right] \notag \\
  &=\ex \!\left[ 
  \g (B_{t})
  \frac{\cosh x\cosh B_{t}}{\cosh ^{2}x+\sinh ^{2}B_{t}}
  \right] . \label{;eqa2}
 \end{align}
 On the other hand, by the symmetry of $\g $, the right-hand side of 
 \eqref{;eqa1} is equal to 
 \begin{align*}
  &\frac{1}{2}\ex \!\left[ 
  \g (B_{t})\left\{ 
  \frac{1}{\cosh (x+B_{t})}+\frac{1}{\cosh (x-B_{t})}
  \right\} 
  \right] \\
  &=\ex \!\left[ 
  \g (B_{t})
  \frac{\cosh x\cosh B_{t}}{\cosh (x+B_{t})\cosh (x-B_{t})}
  \right] . 
 \end{align*}
 Noting the fact that 
 \begin{equation}\label{;eqa3}
  \begin{split}    
   \cosh (x+y)\cosh (x-y)
   &=\frac{1}{2}\left\{ 
   \cosh (2x)+\cosh (2y)
   \right\} \\
   &=\cosh ^{2}x+\sinh ^{2}y
  \end{split}
 \end{equation}
 for any $x,y\in \R $, we compare the last expression with 
 \eqref{;eqa2} to conclude identity \eqref{;eqa1}. 
\end{proof}

We turn to the proof of the implication from \thetag{iii} to \thetag{i}. 
To this end, we prepare the following lemma: 
\begin{lem}\label{;lelem}
 For every $x, b\in \R $, it holds that 
 \begin{align*}
  \int _{\R }
  \frac{dy}{\cosh (x+y)}\,
  \frac{1}{\cosh (2b)+\cosh (2y)}
  =\frac{\pi }{2\cosh b\left( \cosh b+\cosh x\right) }. 
 \end{align*}
\end{lem}

\begin{proof}
 We may assume $|x|\neq |b|$; validity in the case 
 $|x|=|b|$ is verified by passing to the limit. 
 By symmetrization and by the relation 
 $
 \cosh (2b)+\cosh (2y)=2\left( 
 \cosh ^{2}b+\sinh ^{2}y
 \right) 
 $, 
 the left-hand side of the claimed identity is equal to 
 \begin{align*}
  \frac{1}{4}\int _{\R }
  dy\left\{ 
  \frac{1}{\cosh (x+y)}+\frac{1}{\cosh (x-y)}
  \right\} 
  \frac{1}{\cosh ^{2}b+\sinh ^{2}y}, 
 \end{align*}
 which is rewritten, due to relation \eqref{;eqa3}, as 
 \begin{align*}
  &\frac{\cosh x}{2}\int _{\R }dy\,
  \frac{\cosh y}
  {
  \left( \cosh ^{2}x+\sinh ^{2}y\right) \!
  \left( \cosh ^{2}b+\sinh ^{2}y\right) 
  }\\
  &=\frac{\cosh x}{2\left( \cosh ^{2}b-\cosh ^{2}x\right) }
  \int _{\R }dz\left( 
  \frac{1}{\cosh ^{2}x+z^{2}}-\frac{1}{\cosh ^{2}b+z^{2}}
  \right) \\
  &=\frac{\cosh x}{2\left( \cosh ^{2}b-\cosh ^{2}x\right) }
  \left( 
  \frac{\pi }{\cosh x}-\frac{\pi }{\cosh b}
  \right) , 
 \end{align*}
 where we changed the variables with $\sinh y=z$ in the 
 second line. 
 Now the claimed identity follows.  
\end{proof}

We are prepared to finish the proof of \lref{;lequiv}. 

\begin{proof}[Proof of \thetag{iii}\ $\Rightarrow $\ \thetag{i} in \lref{;lequiv}]
 We appeal to the injectivity of Fourier transform. For this purpose, 
 we first observe that 
 \begin{align}\label{;obs3}
  \int _{\R }dx\,\ex \!\left[ 
  \frac{|\f (B_{t})|\cosh B_{t}}{\cosh (2B_{t})+\cosh (2x)}
  \right] <\infty ,  
  &&  
  \int _{\R }dx\,\ex \!\left[ 
  \frac{|\g (B_{t})|}{\cosh (x+B_{t})}
  \right] <\infty . 
 \end{align}
 Indeed, the former observation is immediate from 
 \eqref{;fubi1} and the former condition in \eqref{;intcond} 
 while the latter is clear by the latter condition 
 in \eqref{;intcond}. For an arbitrarily fixed $\xi \in \R $, we 
 integrate both sides of \eqref{;rel1} multiplied by $\cos (\xi x)$ 
 with respect to $x\in \R $. Then by the latter finiteness in \eqref{;obs3} 
 and Fubini's theorem, the right-hand side turns into 
 \begin{align}\label{;rhs2}
  \ex \!\left[ 
  \g (B_{t})\int _{\R }dx\,\frac{\cos (\xi x)}{\cosh (x+B_{t})}
  \right] 
  &=\frac{\pi }{\cosh (\frac{\pi }{2}\xi )}\ex \!\left[ 
  \g (B_{t})\cos (\xi B_{t})
  \right] , 
 \end{align}
 where we used the fact that 
 \begin{align}\label{;fact2}
  \int _{\R }dx\,\frac{\cos (\xi x)}{\cosh x}
  =\frac{\pi }{\cosh (\frac{\pi }{2}\xi )}, 
 \end{align}
 which is verified by standard residue calculus. 
 On the other hand, as for the left-hand side of \eqref{;rel1}, 
 we have 
 \begin{align*}
  &\int _{\R }dx\,\cos (\xi x)\ex \!\left[ 
  \frac{\f (B_{t})}{\cosh B_{t}+\cosh x}
  \right] \\
  &=\frac{2}{\pi }\int _{\R }dx\,\cos (\xi x)
  \ex \!\left[ 
  \f (B_{t})\cosh B_{t}
  \int _{\R }\frac{dy}{\cosh (x+y)}
  \frac{1}{\cosh (2B_{t})+\cosh (2y)}
  \right] \\
  &=\frac{2}{\pi }\int _{\R }dy\,
  \ex \!\left[ 
  \frac{\f (B_{t})\cosh B_{t}}{\cosh (2B_{t})+\cosh (2y)}
  \right] \int _{\R }dx\,\frac{\cos (\xi x)}{\cosh (x+y)}\\
  &=\frac{2}{\cosh (\frac{\pi }{2}\xi )}\int _{\R }dy\,\cos (\xi y)
  \ex \!\left[ 
  \frac{\f (B_{t})\cosh B_{t}}{\cosh (2B_{t})+\cosh (2y)}
  \right] , 
 \end{align*}
 where we used \lref{;lelem} for the second line, Fubini's theorem  
 for the third thanks to the former finiteness in \eqref{;obs3}, 
 and fact \eqref{;fact2} for the fourth. Since the last 
 expression agrees with \eqref{;rhs2} for any $\xi \in \R $ 
 and the function $\g $ is assumed to be symmetric, 
 the injectivity of Fourier transform entails relation \eqref{;rel3}. 
 The proof completes. 
\end{proof}

\begin{rem}\label{;rappend}
 By using \lref{;lelem}, implication \thetag{i} $\Rightarrow$ \thetag{iii} 
 may also be proven in the same manner as in the proof of 
 \thetag{i} $\Rightarrow $ \thetag{ii} given in \sref{;prfmt}. 
\end{rem}


\end{document}